\documentclass{amsart}

\usepackage{amssymb}

\usepackage{graphicx}
\usepackage{algorithm,algorithmic}

\usepackage[latin1]{inputenc}
\usepackage{amssymb,amsmath}
\usepackage{verbatim}
\usepackage{color}
\usepackage{algorithm,algorithmic}
\usepackage{enumerate}
\usepackage{url}

\newtheorem{theorem}{Theorem}[section]
\newtheorem{lemma}[theorem]{Lemma}

\newtheorem{conjecture}[theorem]{Conjecture}

\theoremstyle{definition}

\theoremstyle{remark}
\newtheorem{remark}[theorem]{Remark}

\numberwithin{equation}{section}

\newcommand\Z{\ensuremath{\mathbb Z}}
\newcommand\Q{\ensuremath{\mathbb Q}}\newcommand\R{\ensuremath{\mathbb R}}
\newcommand\C{\ensuremath{\mathbb C}}
\newcommand\Qb{{\overline\Q}}

\newcommand\GL{\operatorname{GL}}

\newcommand\M{\operatorname{M}}

\newcommand\Nm{\operatorname{Nm}}

\newcommand\SL{\operatorname{SL}}

\newcommand{\cH}{\mathcal{H}}
\newcommand{\fp}{{\mathfrak{p}}}
\newcommand{\fc}{{\mathfrak{c}}}
\newcommand{\fn}{{\mathfrak{n}}}
\def\cO{{\mathcal O}}
\newcommand{\mtx}[4]{\left(\begin{matrix}#1&#2\\#3&#4\end{matrix}\right)}

\newcommand{\slims}[4]{\left\{\begin{smallmatrix}#2&#4\\#1&#3\end{smallmatrix}\right\}}

\def\M{\operatorname{M}}

\def\p{\mathfrak p}\def\P{\mathbb P}
\newcommand{\comp}{\begin{picture}(6,5)(-3,-2)\put(0,1){\circle{2}} \end{picture}}\def\circ{\comp}
\def\fN{\mathfrak N}

\newcommand{\ra}{\rightarrow}

\newcommand{\lra}{\longrightarrow}

\begin{document}

\title[Computation of ATR Darmon points]{Computation of ATR Darmon points on non-geometrically modular elliptic curves}


%    author one information
\author{Xavier Guitart}
\address{ Universitat Polit\`ecnica de Catalunya, Barcelona \newline \indent Max Planck Institute for
Mathematics, Bonn}
\curraddr{}
\email{xevi.guitart@gmail.com}

%    author two information
\author{Marc Masdeu}
\address{Columbia University, New York}
\curraddr{}
\email{masdeu@math.columbia.edu}

%    \subjclass is required.
\subjclass[2010]{11G40 (11F41, 11Y99)}

\date{\today}

\dedicatory{}

\begin{abstract}
ATR points were introduced by Darmon as a conjectural construction of algebraic points on
certain elliptic curves for which in general the Heegner point method is not available. So far the only numerical evidence, provided by Darmon--Logan and G\"artner, concerned curves arising as quotients of Shimura curves. In those special cases the ATR points can be obtained from the already existing Heegner points, thanks to results of Zhang and Darmon--Rotger--Zhao.

In this paper we compute for the first time an algebraic ATR point on a curve which is not uniformizable by any Shimura curve, thus providing
the first piece of numerical evidence that Darmon's construction works beyond geometric modularity. To this purpose we
improve the method proposed by Darmon and Logan by removing the requirement that the real quadratic base field be
norm-euclidean, and accelerating the numerical integration of Hilbert modular forms.
\end{abstract}

\maketitle
\section{Introduction}
 Let $F$ be a totally real number field and let $E/F$ be an elliptic curve of conductor $\fN$. Denote by
$L(E/F,s)$ the Hasse--Weil $L$-series attached to $E$, which is known to converge in the half plane
$\Re(s)>3/2$. Let us assume thorough this note that $E$ is \emph{modular}; that is to say, that $L(E/F,s)$ equals the
$L$-series
of a Hilbert modular form  over $F$ of weight $2$ and level $\fN$. Thanks to the
modularity theorems of \cite{Wi},
\cite{BCDT} and \cite{Sk-Wi}, $E$ is known to be modular if either $F=\Q$, or if $[F\colon \Q]>1$ and it satisfies
certain mild
conditions on the reduction type at primes above $3$. The $L$-series $L(E/F,s)$ admits
analytic continuation, and the Birch and Swinnerton-Dyer (BSD) Conjecture predicts that  its
order
of vanishing at $s=1$, called the \emph{analytic rank} of $E/F$, equals the rank of the group of $F$-rational points
$E(F)$. 

In this context,  BSD Conjecture is known to hold in analytic rank $0$ or $1$ provided that $E$ satisfies the following \emph{Jacquet--Langlands
hypothesis}:

\vspace{0.2cm}
\noindent {\bf (JL)} Either $[F\colon \Q]$ is odd or there is a prime $\p$ in $F$ such that  $\operatorname{ord}_{\p}(\fN)$ is
odd.

\begin{theorem}[Gross--Zagier, Kolyvagin, Zhang]\label{theorem: GZKZ}
 Let $E$ be a modular elliptic curve over a totally real number field $F$ satisfying (JL). If
$\operatorname{ord}_{s=1}L(E/F,s)\leq 1$ then 
\[
 \operatorname{ord}_{s=1}L(E/F,s)=\operatorname{rank}(E(F)).
\]
\end{theorem}
In analytic rank $0$, BSD is also known for modular elliptic curves not satisfying (JL), thanks to the work
of Longo \cite{Lo}. However, in analytic rank $1$   (JL) cannot be dispensed with at the moment,  because the 
 construction of  non-torsion points relies on the existence of the so-called Heegner points. Indeed, if $E$ satisfies
(JL) then it 
is \emph{geometrically modular}: there exists a non-constant $F$-homomorphism
\begin{equation}\label{eq: shimura curve parametrization}
 \pi_E\colon \operatorname{Jac}(X)\lra E,
\end{equation}
from the Jacobian of a suitable Shimura curve $X$ defined over $F$ onto  $E$. Shimura curves are endowed with CM points,
which are
defined over ring class fields of quadratic CM extensions $K/F$. The projection of CM points  via $\pi_E$ gives
rise to Heegner points on $E$, whose arithmetic behavior is linked to
the corresponding $L$-series of $E$ thanks to formulas of Gross--Zagier and Zhang. 

On the other hand, if $E$ does not satisfy (JL) then it is not known to be geometrically
modular unless it is a \emph{$\Q$-curve}; i.e., a curve isogenous to all of its Galois conjugates. As a
consequence of Serre's modularity conjecture $\Q$-curves admit $\Qb$-parametrizations from
classical modular Jacobians, and this  has been exploited in \cite{DRZ} in order to construct Heegner points and prove 
BSD in
analytic rank $1$ for some $\Q$-curves not satisfying (JL). But BSD in analytic rank $1$ seems to remain intractable 
for elliptic curves which are not $\Q$-curves and do not satisfy (JL). Indeed, since they are not geometrically modular
the Heegner point method sketched in the previous paragraph cannot be applied in this setting.

The Heegner point construction  constitutes the only known procedure for systematically manufacturing algebraic
non-torsion
points on elliptic curves. However, several \emph{conjectural} constructions have emerged in the last years  under the generic name of
\emph{Stark--Heegner points}, or also \emph{Darmon
points} as the first such construction was introduced in \cite{Da1}. Variants of this initial
construction applying to several different settings have been proposed since then, for instance in \cite{Das},
\cite{Gr},
\cite{LRV}, \cite{Ga-art}, and \cite{GRZ}. The
leitmotif of these methods is the analytic construction of algebraic points on ring class fields of quadratic
extensions $K/F$ which, unlike the classical case, are not  CM.

This note deals with the effective computation of a type of Darmon points known as \emph{ATR
points}, which where introduced in \cite[Chapter VIII]{Da2}. To explain the terminology, recall that a number field 
is
said to be \emph{almost totally real}, or
\emph{ATR} for short, if it has exactly one complex non-real archimedean place.

 Let $K$ be a quadratic ATR extension of the totally real field $F$.
For an ideal $\fc$ of $F$ denote by $R_\fc\subset K$  the order of conductor $\fc$, and $H_\fc$  the ring class field
corresponding to $R_\fc$. Darmon associates to any optimal embedding $\varphi\colon R_\fc\hookrightarrow \M_2(\cO_F)$ a
point $P_\varphi\in E(\C)$, called an \emph{ATR point}, which is 
conjectured to be defined over $H_\fc$. Moreover, by analogy with the formulas of Gross--Zagier and Zhang, its trace down to $K$ is
believed to be non-torsion if and only if $\operatorname{ord}_{s=1}L(E/K,s)=1$.

An algorithm for computing ATR points  in the particular case where  $F$ is  real  quadratic 
and $E$ has conductor $1$
is given in \cite{DL}.  These elliptic curves do not satisfy (JL), so that the Heegner point
construction is not available in general. Both the definition of the points $P_\varphi$  and the conjectures of Darmon concerning them will be
recalled
in Section \ref{sec: ATR points and its computation}. For the moment, it is enough for us to mention that they are the image under
the Weierstrass map $\C/\Lambda_E\ra E(\C)$ of a complex number of the form
\begin{equation}\label{eq: ATR point first definition}
 J_\varphi=\int_{a_1}^{b_1}\int_{c_1}^{d_1} \omega +\cdots + \int_{a_n}^{b_n}\int_{c_n}^{d_n}\omega,
\end{equation}
where $\omega$ is a certain differential $2$-form on the Hilbert modular surface
$\SL_2(\cO_F)\setminus \mathcal H^2$. The limits of integration depend on the embedding $\varphi$, but they are not uniquely
determined: for a given $\varphi$ there are many possible choices for $a_i,b_i,c_i,d_i$. The Fourier series of
$\omega$
is explicitly computable from $E$, and term by term integration of a truncation  leads  to a numerical
approximation to $J_\varphi$. The rate of convergence depends essentially on the
imaginary parts of the limits, and this turns out to be the main computational restriction of this method.

The algorithm outlined above was used in~\cite{DL} to obtain  numerical evidence towards
Darmon's conjectures. More concretely, ATR points on three concrete elliptic curves were computed,
and they were checked to be (up to  a certain numerical precision)  multiples of the corresponding generators of the
Mordell--Weil groups. Calculations of the same kind were performed in \cite{Ga-th} for one more curve.
However, computational limitations restricted them to elliptic curves which all happen to be geometrically modular. In
this case, the BSD Conjecture implies that ATR points in these curves should be related to the already existing Heegner
points. Actually, in the case of
$\Q$-curves (as they are all examples considered in ~\cite{DL}) a precise relation between Heegner and ATR points is
conjectured in~\cite[\S 4.2]{DRZ}.

In this note we speed up the algorithm devised by Darmon and Logan by improving its two main bottlenecks. Namely,  the
computation of integrals of Hilbert modular forms, and the determination of limits $a_i,b_i,c_i,d_i$ in \eqref{eq: ATR
point first definition} having the highest imaginary part possible. This allows us to gather more numerical evidence in
support of  Darmon's conjectures, by calculating ATR points on elliptic curves which  were not computationally
accessible using the algorithm in~\cite{DL}.  In particular, we have been able to compute for the first time
an  ATR point of infinite order on a \emph{non-geometrically
modular} elliptic curve.

More concretely, the contents of the article are as follows. In
Section \ref{sec: ATR points and its computation} we review the definition of the points $P_\varphi$ and Darmon's
conjectures on their arithmetic properties. We also sketch the algorithm of Darmon and Logan for their explicit
computation.

 In Section \ref{sec: Integrals of Hilbert Modular
forms} we present the algorithm for speeding up the computation of integrals of Hilbert modular forms. The idea is to
use
the fact that the limits are invariant under the group $\SL_2(\cO_F)$ in order to transform the given integral into a sum of integrals whose  limits
are uniformly bounded away from the real axis, the bound depending only on  $F$. It is worth remarking that this
algorithm does not exploit any particular property of the integrals involved in ATR points, and therefore it may be of
independent interest for
computing integrals of Hilbert modular forms in other contexts.

In Section \ref{sec: Dependence on the continued
fractions} we comment on a trick that can sometimes
accelerate the computation of ATR points. The procedure for computing the limits $a_i,b_i,c_i,d_i$ in \eqref{eq: ATR point
first definition} for an embedding $\varphi$ involves at some point the calculation of a continued fraction expansion
of
an element in $F$.  We exploit the non-uniqueness of continued fractions to attach to a
given $\varphi$ limits $a_i,b_i,c_i,d_i$  with as high imaginary part as
possible.

Finally, in  Section \ref{sec: Numerical verification of Darmon's conjecture} we use Darmon--Logan's algorithm
together with the improvements of sections \ref{sec: Integrals of Hilbert Modular
forms} and \ref{sec: Dependence on the continued
fractions} to enlarge the pool of elliptic curves on which Darmon's conjectures have been numerically tested. 
Arguably, the most interesting
among them is the curve $E_{509}$ given by the
equation
\begin{equation*}
 y^2-xy-\omega y=x^3+(2+2w)x^2+(162+3w)x+(71+34\omega),\ \ \ \omega=\frac{1+\sqrt{509}}{2},
\end{equation*}
because it is not a $\Q$-curve.  We
have computed an ATR point
corresponding to the field $K=F(\sqrt{9144\omega + 98577})$, and we have numerically checked that it coincides  with a
multiple of the Mordell--Weil generator of $E_{509}(K)$. Since $E_{509}$  is not geometrically modular, such point
does not seem to be explained by the presence of Heegner points. This gives experimental evidence that
Darmon's construction  leads to algebraic points that are genuinely new, not attainable by classical
methods.

Finally, it is worth mentioning that ATR points are also the base of an algorithm by L. Demb\'el\'e  \cite{De} for
computing equations of elliptic curves with everywhere good reduction  attached to Hilbert
modular forms of level $1$. The authors hope that the algorithm presented in this note can be useful for this purpose,
and that it may lead in the future to a systematic  computation of such equations using Demb\'el\'e's method.
\subsubsection*{Acknowledgments} The authors are thankful to  Henri Darmon and Victor Rotger for initially
suggesting the problem and for many helpful conversations, and to the anonymous referee for many valuable
comments and suggestions. They are grateful to the Max Planck Institute for Mathematics for the hospitality and
financial support, and for making available its computational resources, crucially needed for
part of this note. Part of this work was also carried out in the facilities of Centro de Ciencias de Benasque
Pedro Pascual during the Summer of 2011. The authors received  financial support from DGICYT Grant 
MTM2009-13060-C02-01 and from 2009 SGR 1220.

\section{Computation of ATR points}\label{sec: ATR points and its computation}

Let $F$ be a real quadratic number field of discriminant $D$ and narrow class number $1$. Write $\cO_F$ for  its  ring of
integers and set
$\Gamma=\SL_2(\cO_F)$. We denote by $v_0,v_1$  the embeddings of $F$ into $\R$. For an element $x\in F$ we may write
$x_i$ instead of $v_i(x)$, and  $|x|$ for the norm of $x$, given by $x_0x_1$. Recall that $\Gamma$ acts discretely on
$\mathcal{H}^2$ via $v_0\times v_1$. The analytic variety $\Gamma\setminus
\mathcal{H}^2$ can be compactified by adding one cusp, which gives rise to the Hilbert modular surface $X$ attached to
$\Gamma$.

Let $K/F$ be a quadratic ATR extension. For an ideal $\fc$ of $F$ we denote by $H_\fc$ the ring class field corresponding
to the order $R_\fc$ of conductor $\fc$ of $K$. We assume without loss of
generality that $v_0$ extends to a
complex place of $K$ and $v_1$ extends to a pair of real places of $K$. We fix  an extension of $v_0$ to
$\overline{\Q}\subset \C$, which we use to to identify $K$ and its extensions with subfields of $\C$. Recall that an
embedding
of $\cO_F$-algebras $\varphi\colon R_\fc\hookrightarrow M_2(F)$ is
said to be \emph{optimal}  if $\varphi(K)\cap M_2(\cO_F)=\varphi(R_\fc).$ We denote by $\mathcal{E}_\fc$ the set of such
optimal embeddings.

Let $E/F$ be an elliptic curve of conductor $1$. In this section we review Darmon's construction, which attaches to
each optimal embedding $\varphi$ a point $P_\varphi\in E(\C)$ that conjecturally belongs to $E(H_\fc)$. There are
several
equivalent
ways of defining $P_\varphi$. For instance,  a nice geometric definition in terms of a non-algebraic analogue to the Abel--Jacobi map
is given
in \cite[\S 2.1]{DRZ}. However, for computational purposes the original definition of \cite[Chapter VIII]{Da2},  or
rather the subsequent refinement of \cite{DL} are better suited. Key to the approach of \cite{DL} is the
definition of certain \emph{semi-definite integrals} of Hilbert modular forms, whose existence and main properties we
will
take as a black box.

\subsection{Semi-definite integrals of Hilbert modular forms} Let $f\in S_2(\Gamma)$ be a Hilbert modular form. Recall
that $f$ has a Fourier expansion indexed by totally positive elements of $\cO_F$. Actually, the Fourier coefficient
corresponding to $n\in \cO_F^+$ only depends on the ideal $(n)$ generated by $n$, and the expansion is of the form
\begin{equation}\label{eq: fourier expansion}
 f(\tau_0,\tau_1)=\sum_{n\in \cO_F^+} a_{(n)}e^{2\pi i (\frac{n_0}{\delta_0}\tau_0+\frac{n_1}{\delta_1}\tau_1)},\ \
\tau_0,\tau_1\in\mathcal H,
\end{equation}
where $\delta_i=v_i(\delta_F)$ and $\delta_F$ is a totally positive generator of the different ideal of $F$. Let us
assume from now on that all Fourier
coefficients $a_{(n)}$ are rational numbers. The reader can refer to~\cite[\S 2.4]{GRZ} for the definition of ATR points
when
the Fourier coefficients generate a number field of degree $>1$, in which case they belong to some higher
dimensional modular abelian variety. 

The differential form
\[
 \omega_f=\frac{(2\pi i)^2}{\sqrt{D}} f(\tau_0,\tau_1)d\tau_0d\tau_1
\]
is invariant under the action of $\Gamma$ and extends to a holomorphic form at the cusp, thus defining a
holomorphic $2$-form on $X$. The expansion in Equation~\eqref{eq: fourier expansion} is useful for computing
integrals of $\omega_f$. Indeed,  for $x_0,x_1,y_0,y_1\in
\mathcal H$ we have that
\begin{small}
\begin{equation}\label{eq: integrals of w_f}
 \int_{x_0}^{y_0}\int_{x_1}^{y_1}\omega_f=\sqrt{D}\sum_{n \in \cO_F^+}\frac{a_{(n)} }{|n|}\left( e^{2\pi i
\frac{n_0}{\delta_0}
y_0}-e^{2\pi i \frac{n_0}{\delta_0}
x_0}\right)
\left( e^{2\pi i \frac{n_1}{\delta_1}
y_1}-e^{2\pi i \frac{n_1}{\delta_1}
x_1}\right).
\end{equation}
\end{small}
In the definition of ATR points the key role is not played by  $\omega_f$ but instead by the 
non-holomorphic differential
\begin{equation*}
 \omega_f^+=\frac{(2\pi i)^2}{\sqrt{D}}\Big( f(\tau_0,\tau_1)d\tau_0d\tau_1+
f(u_0\tau_0,u_1\overline{\tau}_1)d(u_0\tau_0)d(u_1\overline{\tau}_1) \Big),
\end{equation*}
where $u$ is a fundamental unit in $F$ such that $u_0>1$ and $u_1<-1$. The differential form $\omega_f^+$ is also
$\Gamma$-invariant, and it follows from the definition that it is invariant under the action of the matrix
$\tilde\gamma_u=\mtx u 0 0 1$. Therefore, if we let
\[
 \tilde \Gamma=\{\gamma \in \GL_2(\cO_F) \ \colon \ v_0(\det(\gamma))>0\}
\]
we find that
\begin{equation*}
 \int_{x_0}^{y_0}\int_{x_1}^{y_1}\omega_f^+=\int_{\gamma x_0}^{\gamma y_0}\int_{\gamma x_1}^{\gamma y_1}\omega_f^+, \
\ \ \text{for all } \gamma\in \tilde\Gamma.
\end{equation*} 
Here $\gamma$ acts on the outer limits (reps. inner limits) through $v_0$ (resp. $v_1$).

Let
\[
 \Lambda_f^+=\left\{\int_\delta \omega_f^+\ \colon \delta\in H_2(X,\Z) \right\}\subset \C
\]
be the period lattice of $\omega_f^+$.
\begin{theorem}[Darmon--Logan]
 There exists a unique map
\begin{equation*}
 \begin{array}{ccc}
\mathcal{H}\times \P^1(F)\times \P^1(F) & \longrightarrow & \C/\Lambda_f^+\\
(\tau_0,c_1,c_2) & \longmapsto & \int^{\tau_0}\int_{c_1}^{c_2} \omega_f^+
 \end{array}
\end{equation*}
satisfying the following properties:
\begin{enumerate}[(i)]
 \item\label{enum:p1} $\int^{\gamma\tau_0}\int_{\gamma c_1}^{\gamma c_2} \omega_f^+=\int^{\tau_0}\int_{c_1}^{c_2}
\omega_f^+$ for all $\gamma\in \tilde\Gamma$,
\item\label{enum:p2} $\int^{\tau_0}\int_{c_1}^{c_2} \omega_f^+ + \int^{\tau_0}\int_{c_2}^{c_3}
\omega_f^+=\int^{\tau_0}\int_{c_1}^{c_3} \omega_f^+$,
\item\label{enum:p3} $\int^{y}\int_{c_1}^{c_2} \omega_f^+-\int^{x}\int_{c_1}^{c_2}
\omega_f^+=\int_x^{y}\int_{c_1}^{c_2} \omega_f^+$.
\end{enumerate}
\end{theorem}
For the existence of such map we refer the reader to \cite[\S 1]{DL}. Uniqueness is proved in \cite[\S 4]{DL},
and it follows from repeated application
of properties (\ref{enum:p1}), (\ref{enum:p2}) and (\ref{enum:p3}). Since this also leads to an algorithm for
computing the map, we review the proof in the next section.

\subsection{Computation of semi-definite integrals via continued fractions}\label{subsection: continued fractions trick}
Given $b_0,b_1,\dots,b_n\in \cO_F$ the
(finite) \emph{continued fraction}
$[b_0,b_1,\dots,b_n]\in F$ is defined, inductively, as
\[
 [b_0]=b_0, \ [b_0,b_1]=b_0+\frac{1}{b_1},\dots, [b_0,b_1,\dots,b_n]=[b_0,[b_1,\dots,b_n]].
\]
Let $[b_0,b_1,\dots,b_k] =\frac{p_k}{q_k}$, with $p_k,q_k\in \cO_F$ coprime. It is well known that 
\begin{equation}\label{eq: convergents}
p_kq_{k-1}-p_{k-1}q_k=(-1)^{k-1}. 
\end{equation}
 Since $\cO_F^\times$ is infinite and $F$ has trivial
class number, by a result of Cooke--Vaserstein \cite{cooke} every element $c\in F$ can be written as a finite continued
fraction. See \cite{GM} and \S \ref{sec: Dependence on the continued fractions} for an effective version of this result.

Two cusps $c_1,c_2\in \P^1(F)$ are said to be \emph{adjacent} if $c_1=\gamma\cdot 0$ and $c_2=\gamma\cdot \infty$ for
 some $\gamma\in \Gamma$. One can join every cusp $c\in F$ with $\infty$ by a sequence of adjacent cusps. Indeed, if
$c=[b_0,b_1,\dots,b_n]$ then the sequence
\[
 \infty,\frac{p_0}{q_0},\frac{p_1}{q_1},\dots,\frac{p_n}{q_n}=c
\]
has this property thanks to \eqref{eq: convergents}. Using this fact and property~(\ref{enum:p1}) every integral
$\int^{\tau_0}\int_{c_1}^{c_2}\omega_f^+$ can be written as a sum of integrals of the form
$ \int^{\tau_0}\int_{\gamma\cdot 0}^{\gamma\cdot\infty}\omega_f^+$. Thanks to~(\ref{enum:p1}) they are of the
form $\int^\tau\int_{0}^\infty\omega_f^+$, and can be computed as follows:
\begin{align*}
 \int^\tau\int_0^\infty\omega_f^+&=\int^\tau\int_0^1\omega_f^++\int^\tau\int_1^\infty\omega_f^+\\
&=\int^{-\frac{1}{\tau}}\int_\infty^{-1}\omega_f ^++\int^{\tau-1}\int_0^\infty\omega_f^+\\
&=\int_{1-\frac{1}{\tau}}^{\tau-1}\int_{0}^{\infty}\omega_f^+.
\end{align*}
Integrals of the form $\int_{\tau_1}^{\tau_2}\int_0^\infty\omega_f^+$, with $\tau_1,\tau_2\in\mathcal{H}$, can be
computed by taking $\tau_3\in \mathcal{H}$ and writing
\begin{equation}\label{eq: from semidefinite to definite}
 \int_{\tau_1}^{\tau_2}\int_0^\infty\omega_f^+=\int_{\tau_1}^{\tau_2}\int_0^{\tau_3}\omega_f^++\int_{\tau_1}^{\tau_2}
\int_{\tau_3}^\infty\omega_f^+=\int_{-\frac{1}{\tau_1}}^{-\frac{1}{\tau_2}}\int_\infty^{-\frac{1}{\tau_3}}\omega_f^++\int_{\tau_1}^{\tau_2}
\int_{\tau_3}^\infty\omega_f^+,
\end{equation}
for which formula \eqref{eq: integrals of w_f} can be used.

\subsection{Definition of ATR points}

 Under the assumption that $E$ is modular there exists a Hilbert modular newform $f_E\in
S_2(\Gamma)$ such that $L(E/F,s)=L(f_E,s)$. The Fourier expansion of $f_E$ can be explicitly computed as follows. For
 a prime ideal $\fp\subset F$, let $a_\fp:=|\fp|+1-\# E(\cO_K/\fp)$, where $|\fp|$ denotes $\# \cO_F/\fp$. For arbitrary
ideals $\fn\subset F$ the integer $a_\fn$ is defined by means of the identity
\[
 \sum_{\fn\subset F} a_\fn|\fn|^{-s}=\prod_{\fp \ \text{prime}} \left( 1-a_\fp|\fp|^{-s} +|\fp|^{1-2s} \right)^{-1},
\]
and the Fourier expansion of $f_E$ is given by
\begin{equation}
 f_E(\tau_0,\tau_1)=\sum_{n\in \cO_F^+} a_{(n)}e^{2\pi i (\frac{n_0}{\delta_0}\tau_0+\frac{n_1}{\delta_1}\tau_1)},\ \
\tau_0,\tau_1\in\mathcal H.
\end{equation}

 Let $\varphi\colon R_\fc\hookrightarrow M_2(F)$ be an optimal embedding. Since $K$ is ATR the group
\[
 \Gamma_\varphi=\{ \gamma \in \varphi(\mathcal{O}_F)\ \colon \ \det(\gamma)=1\}
\]
has rank $1$; let $\gamma_\varphi$  be one of its generators. Since $v_0$ extends to a complex place of $K$,
the action of 
$K^\times$ on $\mathcal H$  by means of $v_0\circ \varphi$ has a single fixed point $\tau_0$. Let $J_\varphi$ be the
quantity in $\C/\Lambda_f^+$ defined as
 \begin{equation}\label{eq: definition of ATR point} J_\varphi=\int^{\tau_0}\int_{\infty}^{\gamma_\varphi
\infty}\omega_f^+.\end{equation}

Let $\omega_E\in H^0(E,\Omega^1)$ be a differential which extends to a smooth differential on the N\'eron model of $E$
over $\cO_F$. Let 
\[
 \Lambda_i=\left\{ \int_Z v_i(\omega_{E}) \ \colon \ Z \in H_1(E_j,\Z) \right\}
\]
and let $\lambda_j^+$ (resp. $\lambda_j^-$) be a generator of $\Lambda_j\cap \R$ (resp. $\Lambda_j\cap i\R$). 

\begin{conjecture}[Oda]\label{conj: Oda}
 There exists a integer $c$ such that $\frac{c}{\lambda_1^+}\Lambda_f^+\subseteq \Lambda_0$.
\end{conjecture}

Granting Conjecture  \ref{conj: Oda} and denoting by $\eta\colon \C/\Lambda_0\ra E_0(\C)$  the Weierstrass
parametrization map, the ATR point $P_\varphi$ is defined as
\[
 P_\varphi=\eta(\frac{c}{\lambda_1^+}\cdot J_\varphi) \in E_0(\C).
\]
The group $\tilde \Gamma$ acts by conjugation on $\mathcal{E}_\fc$, and the set of equivalence classes
$\mathcal{E}_\fc/\tilde\Gamma$ has cardinal $[H_\fc \colon K]$ (cf. \cite[\S 8.5]{Da2}).
\begin{conjecture}[Darmon]
The point $P_\varphi$ belongs to $E_0(H_\fc)$. Moreover, the point
\[
 P_K=\sum_{\varphi\in \mathcal{E}_\fc/\tilde\Gamma} P_\varphi
\]
 belongs to
$E_0(K)$ and it is non-torsion if and only if $\operatorname{ord}_{s=1}L(E/K,s)=1$.
\end{conjecture}

Note that under the assumptions of this section, namely the conductor of $E/F$ is $1$ and $K/F$ is ATR, the
$L$-function $L(E/K,s)$ has sign $-1$. Thus the condition  $\operatorname{ord}_{s=1}L(E/K,s)=1$ is equivalent to
$L'(E/K,1)\neq 0$, which can be numerically checked.

\section{Integration of Hilbert Modular forms}\label{sec: Integrals of Hilbert Modular forms}

When evaluating \eqref{eq: integrals of w_f} one can collect the elements $n\in \cO_F^+$ modulo powers of $u^2$, because
the Fourier coefficient corresponding to $n$ only depends on the ideal $(n)$. The sum corresponding to $n\cdot \langle
u^2\rangle$ is then
\[
S_n= \sqrt{D}\sum_{k=-\infty}^{\infty}\frac{ a_{(n)}}{|n|}\left( e^{2\pi i \frac{n_0}{\delta_0}u_0^{2k}
y_0}-e^{2\pi i \frac{n_0}{\delta_0}{u_0^{2k}}
x_0}\right)
\left( e^{2\pi i \frac{n_1}{\delta_1}{u_0^{-2k}}
y_1}-e^{2\pi i \frac{n_1}{\delta_1}{u_0^{-2k}}
x_1}\right).
\]
For a fixed $k$, the modulus of each exponential term (after multiplying out the brackets)
is of the form
\begin{equation}\label{eq: exponential term}
 e^{-2\pi(\frac{n_0}{\delta_0}ru_0^{2k}+\frac{n_1}{\delta_1}su_0^{-2k})}, \ \ \text{with} \ 
r\in\{\Im(y_0),\Im(x_0)\},s\in\{\Im(y_1),\Im(x_1)\}.
\end{equation}
It is easy to see that \eqref{eq: exponential term} has a single maximum, when viewed as a function of $k$, and that
the maximum value afforded by the four exponential terms is bounded by
\[
M_n= e^{\frac{-4\pi\sqrt{|n|}\epsilon}{\sqrt{D}} },
\]
where $\epsilon = \epsilon(x_0,y_0,x_1,y_1)$ is defined by:
\begin{equation}\label{eq: def de epsilon}
\epsilon(x_0,y_0,x_1,y_1)^2=\min\{\Im(x_0)\Im(x_1),\Im(x_0)\Im(y_1),\Im(y_0)\Im(y_1),\Im(y_0)\Im(x_1)\}.
\end{equation}
Moreover, one easily checks that $S_n$ is dominated  by a geometric series and that
\[S_n\leq \frac{2\sqrt{ D} |a_{(n)}|M_n}{|n|},\quad \text{for } |n|>>0.\]
This estimate allows us to know a priori how many terms need to be considered in order to obtain a prescribed accuracy. We observe that the speed of convergence of expression \eqref{eq: integrals of w_f} depends on the limits
$x_0,y_0,x_1,y_1$ through the quantity $\epsilon(x_0,y_0,x_1,y_1)$. The main result of this section is the following.
\begin{theorem}\label{th: main}
 There exists a constant $\epsilon_F$, which depends only on $F$, such that for every $(x_0,y_0,x_1,y_1)\in
\mathcal{H}^4$ and for every $\epsilon_0<\epsilon_F$ the integral $\int_{x_0}^{y_0}\int_{x_1}^{y_1}\omega_f$ can be
expressed as
\begin{equation}\label{eq: eq del main theorem}
 \int_{x_0}^{y_0}\int_{x_1}^{y_1}\omega_f=\int_{a_1}^{b_1}\int_{c_1}^{d_1}\omega_f +\cdots +
\int_{a_n}^{b_n}\int_{c_n}^{d_n}\omega_f,
\end{equation}
with $\epsilon(a_i,b_i,c_i,d_i)\geq \epsilon_0$ for all $i=1,\dots ,n$.
\end{theorem}

 Observe that in the integrals considered in Theorem \ref{th: main} the four limits of integration lie in $\cH$. By
 \eqref{eq: from semidefinite to definite}  integrals of the form
$\int_{x_0}^{y_0}\int_{x_1}^{\infty}\omega_f^+$ are involved in the computation of ATR points. One
can choose any $y_1\in\cH$ whose imaginary part is large enough to satisfy that
\[
 \int_{x_0}^{y_0}\int_{x_1}^{\infty}\omega_f^+=\int_{x_0}^{y_0}\int_{x_1}^{y_1}\omega_f^++\int_{x_0}^{y_0}\int_{y_1}^{
\infty}\omega_f^+,\ \ \text{ with  } \Im(x_0)\Im(y_1),\Im(y_0)\Im(y_1)>\epsilon_F,
\]
and apply Theorem \ref{th: main} to the integral $\int_{x_0}^{y_0}\int_{x_1}^{y_1}\omega_f^+$.

We devote the rest of the section to prove Theorem \ref{th: main}, which as we will see can be made effective and
algorithmic. 

It will be useful for us to regard $F$ as a subset of $\R^2$ by means of $v_0\times v_1$. Let $\|\cdot\|$ denote the
norm on $\R^2$ given by
\[
\|(x_0,x_1)\| = \max\{|x_0|,|x_1|\}.
\]
The basic ingredient in the proof of \ref{th: main} is the
following classical result.
\begin{lemma}\label{lemma: freitag}
  There exists a constant $C_F$, only depending on $F$, such that for each $x\in\R^2$ and for each $0<\delta<1$ there
are elements
$c,d\in\cO_K$, with $c\neq 0$, such that
\[
\| cx+d\|\leq \delta,\quad \|c\|\leq \frac {C_F}{\delta}.
\]
\end{lemma}
\begin{proof}
  This is \cite[Lemma 3.6]{Fr}.  We rewrite the proof in an algorithmic fashion in order to give an approximation to 
$C_F$ (see Remark \ref{remark: explicit constants} below), as
well as to give an algorithm to find the elements $c$ and $d$. 

Consider a fundamental parallelogram $P$ for $\cO_F$ as a
subgroup of $\R^2$. Let $U_1,\ldots ,U_N$ be boxes of side $\delta$ that cover $P$. It is easy to see that there
is a constant $N'$, depending only on $F$, such that $N$ can be taken to be $\leq N'/\delta^2$.

For each positive real $r$, consider the set
\[
S_F(r) = \left\{ c \in \cO_F ~|~ \|c\|< r\right\}.
\]

A well-known result in Ehrhart theory (see e.g.~\cite[Theorem~2.9]{Matthias-Sinai}) implies that there exists a constant
$C_F>0$, which depends on $F$ but not on $\delta$, such that
\[
\# S_F\left(\frac{C_F}{2\delta}\right) > \frac{N'}{\delta^2}\geq  N.
\]
 Consider now an ordering $\{c_n\}_{n\geq 1}$ of the elements of $S_F(C_F/(2\delta))$. For each of the $c_n$, find $d_n\in\cO_F$ such that
\[
c_nx+d_n\in P,
\]
and set $i(n)$ to be the integer such that $c_nx+d_n\in U_{i(n)}$. By the pigeonhole principle the sequence
$\{i(n)\}_{n}$ will have a repetition, say $i(n_1)=i(n_2)$.  Therefore
\[
\|(c_{n_1}-c_{n_2})x+(d_{n_1}-d_{n_2})\| <\delta,
\]
and
\[
\|c_{n_1}-c_{n_2}\|\leq \|c_{n_1}\|+\|c_{n_2}\|\leq \frac{C_F}{2\delta}+\frac{C_F}{2\delta} = \frac{C_F}{\delta}.
\]
The seeked elements are thus $c=c_{n-1}-c_{n-2}$ and $d=d_{n-1}-d_{n-2}$.
\end{proof}
\begin{remark}\label{remark: explicit constants}
  In our applications, the parameter $\delta$ will be small enough so that the quantity $N$ in the above proof can be
approximated by $\operatorname{area}(P)/\delta^2$. Also, \cite[Theorem~2.9(b)]{Matthias-Sinai} gives in this case that
\[
\# S_F\left(\frac{r}{2\delta}\right)\geq \frac{r^2}{4\delta^2}\frac{4}{\operatorname{area}(P)} +1.
\]
Therefore, it is enough for $\# S_F(\frac{r}{2\delta})$ to be larger than $N$ that
\[
\frac{r^2}{4\delta^2}\frac{4}{\operatorname{area}(P)}\geq \frac{\operatorname{area}(P)}{\delta^2}.
\]
This yields an approximate  value of $C_F \approx \operatorname{area}(P)$, which is good enough for our purposes.
Note that this area is easily calculated: if $\cO_F=\Z \oplus \Z w$, then $\operatorname{area}(P)=|w_0-w_1|$.
\end{remark}
We define $$\epsilon_F=\frac{u_0}{C_F(1+u_0^2)}.$$ In the next lemmas we prove that this can, indeed, be taken as the 
 constant $\epsilon_F$ of Theorem \ref{th: main}.

\begin{lemma}\label{lemma: moving to good domain}
Let $z=(z_0,z_1)$ be an element in $\cH\times \cH$. There exists a matrix $\gamma\in \Gamma$ such
that 
\[
  \left( \Im(\gamma_0 z_0)\Im(\gamma_1 z_1)\right)^{1/2}\geq \epsilon_F.
\]
\end{lemma}
\begin{proof}
Let $z_j=r_j+is_j$ and let $r=(r_0,r_1)$.  If $s_0s_1\geq\epsilon_F^2$ one may take $\gamma = 1$ and there result is
obvious, so assume from now on that $s_0s_1< \epsilon_F^2$.  By Lemma \ref{lemma: freitag} for each $0<\delta<1$ there
exist $c',d'\in
\cO_F$ with $c'\neq 0$ and such that 
\[
 \| c'r+d' \| \leq \delta,\quad  \| c'\| \leq \frac{C_F}{\delta}.
\]
We have that
\begin{equation*}
\left( (c'_0r_0+d'_0)^2+c_0'^2s_0^2\right)\left( (c'_1r_1+d'_1)^2+c_1'^2s_1^2\right)\leq \left(
\delta^2+\frac{C_F^2s_0^2}{\delta^2} \right) \left(
\delta^2+\frac{C_F^2s_1^2}{\delta^2} \right).
\end{equation*}
Let $g=\gcd(c',d')$ and let $n=\Nm_{K/\Q}(g)$ . If we let $c=c'/g$ and $d=d'/g$ then
\begin{align*} 
\left( (c_0r_0+d_0)^2+c_0^2s_0^2\right)\left( (c_1r_1+d_1)^2+c_1^2s_1^2\right)&\leq \frac{1}{n^2}\left(
\delta^2+\frac{C_F^2s_0^2}{\delta^2} \right) \left(
\delta^2+\frac{C_F^2s_1^2}{\delta^2} \right)\\ &\leq\left(
\delta^2+\frac{C_F^2s_0^2}{\delta^2} \right) \left(
\delta^2+\frac{C_F^2s_1^2}{\delta^2} \right).
\end{align*}
Since $\gcd(c,d)=1$ there exists  $\gamma\in \Gamma$ having $(c,d)$ as bottom row and
\begin{align*}
  \Im(\gamma_0 z_0)\Im(\gamma_1
z_1)&=\frac{s_0}{|c_0z_0+d_0|^2}\frac{s_1}{|c_1z_1+d_1|^2}\\
&=\frac{s_0s_1}{\left((c_0r_0+d_0)^2+c_0^2s_0^2\right)\left((c_1r_1+d_1)^2+c_1^2s_1^2\right)}\\
&\geq\frac{s_0s_1}{\left(\delta^2+\frac{C_F^2s_0^2}{\delta^2}\right)\left(\delta^2+\frac{C_F^2s_1^2}{\delta^2}\right)}.
\end{align*}
We choose $\delta$ that maximizes this expression. The optimal value for $\delta$ turns out to be
\[
\delta=(C_F^2s_0s_1)^{1/4}.
\]
Note that $\delta<1$ since $s_0s_1<\epsilon_F^2$. With this value of $\delta$ the above inequality gives
\[
\left( \Im(\gamma_0 z_0)\Im(\gamma_1 z_1)\right)^{1/2}\geq \frac{1}{2C_F}\frac{G(s_0,s_1)}{A(s_0,s_1)},
\]
where $G(s_0,s_1)$ and $A(s_0,s_1)$ are the geometric and arithmetic means, respectively. Of course, for this quantity
to be not too small we should ensure that the ratio $G/A$ is not too small. That is, $s_0$ and $s_1$ should be
close. This can be done by the action of the matrix $\gamma_u=\mtx u 0 0{u^{-1}}$, which
guarantees that:
\begin{equation}\label{eq: ratio between the imaginary parts}
u_0^{-2} \leq \frac{s_1}{s_0}\leq u_0^2.
\end{equation}
Therefore one obtains (the worst case is when the ratio is at any extreme):
\begin{equation}\label{eq: inequality AG}
\frac{G(s_0,s_1)}{A(s_0,s_1)}\geq \frac{2u_0}{1+u_0^2}, 
\end{equation}
and the result follows.
\end{proof}

\begin{remark}\label{rk: the actual eps_F}
 A similar argument to that of Lemma  \ref{lemma: moving to good domain} shows that for any $(z_0,z_1)\in \cH\times\cH$
there exists $\gamma\in\tilde\Gamma$ such that
\begin{align*}
\left( \Im(\gamma_0 z_0)\Im(\gamma_1 z_1)\right)^{1/2}\geq \frac{\sqrt{u_0}}{C_F(1+u_0)}.
\end{align*}
Indeed,
in this case one can improve \eqref{eq: ratio between the imaginary parts} to
\[
 u_0^{-1} \leq \frac{y_1}{y_0}\leq u_0
\]
by using the action of $\tilde\gamma_u=\mtx u 0 0 1\in \tilde\Gamma$. We will apply this remark to the computation of
ATR points, because the integrals $\int_{x_0}^{y_0}\int_{x_1}^{y_1}\omega_f^+$ are $\tilde\Gamma$-invariant, and then we
can take $\epsilon_F$ in Theorem~\ref{th: main} to be 
\begin{equation}\label{eq: true epsilon_F}
 \epsilon_F=\frac{\sqrt{u_0}}{C_F(1+u_0)}.
\end{equation}

\end{remark}
 
Given $x,y \in \cH$ denote by $\rho(x,y)$ the geodesic in $\cH$ joining $x$ and $y$. We also let $d(x,y)$ be the
\emph{hyperbolic distance} between $x$ and $y$. This distance is invariant under the action of $\SL_2(\R)$, and it is
given
by the formula
\begin{equation*}
 \cosh d(x,y)=1+\frac{\mid x-y \mid^2}{2\Im(x)\Im(y)}.
\end{equation*}
Observe that 
\begin{equation}\label{eq: inequality distance}
 d(x,y)\geq \cosh^{-1}\left(1+\frac{(\Im(x)/\Im(y)-1)^2}{2\Im(x)/\Im(y)} \right).
\end{equation}

For $x_0,y_0,x_1,y_1\in \cH$ we denote by $\slims{x_0}{y_0}{x_1}{y_1}$ the image under the quotient map
$\cH^2\sqcup\{\infty\}\ra \Gamma\setminus(\cH^2\sqcup\{\infty\})$ of the region $\rho(x_0,y_0)\times
\rho(x_1,y_1)\subset
\cH\times \cH$. Since  $\SL_2(\R)$ acts by isometries on $\cH$ we have that
$\slims{x_0}{y_0}{x_1}{y_1}=\slims{\gamma_0x_0}{\gamma_0y_0}{\gamma_1x_1}{\gamma_1y_1}$ for all $\gamma\in \Gamma$ and
the \emph{area} of
$\slims{x_0}{y_0}{x_1}{y_1}$ is given by $d(x_0,y_0)d(x_1,y_1)$.

 For a given $\epsilon_0<\epsilon_F$,
let $\epsilon_1:=\sqrt{\epsilon_0\epsilon_F}$. This  auxiliary quantity will be used in the proof of Theorem
\ref{th: main}. 
We also define quantities $d_0$, $d_1$ and $d_\text{min}$ as
\[
 d_i=\cosh^{-1}\left( 1+\frac{((\epsilon_F/\epsilon_i)^2-1)^2}{2(\epsilon_F/\epsilon_i)^2} \right) \text{ for
}i\in\{0,1\}, \text{ and }
d_\text{min}=\min\{d_1,d_2\}.
\]

\begin{lemma} \label{lemma: breaking the integrals}
Let $x_0,y_0,x_1,y_1\in\cH$ be such that $\epsilon(x_0,y_0,x_1,y_1)<\epsilon_0<\epsilon_F$.
Then $\int_{x_0}^{y_0}\int_{x_1}^{y_1}\omega_f$
can be written as a sum  of either two or three integrals along sub-regions of $\slims{x_0}{y_0}{x_1}{y_1}$ with zero-measure intersection, in which the first one is of the form
$\int_{x'_0}^{y'_0}\int_{x'_1}^{y'_1}\omega_f$
with $x_0',y_0',x_1',y_1'$ satisfying the following properties:
\begin{enumerate}[(1)]
\item $\epsilon(x'_0,y'_0,x'_1,y'_1)\geq \epsilon_0$
\item  for each $i\in\{0,1\}$ either $(x_i',y_i')=(x_i,y_i)$ or $d(x'_i,y'_i)\geq d_\text{min}$;
\end{enumerate}
\end{lemma}

\begin{proof}
Actually we will prove that $x_0',y_0',x_1',y_1'$ can be chosen to satisfy 
  \begin{equation}\label{eq: more restrive contitions on de x_i'}\Im(x'_0)\Im(x'_1)\geq {\epsilon_F^2},\ 
   \Im(x'_0)\Im(y'_1)\geq {\epsilon_1^2},\ 
   \Im(y'_0)\Im(y'_1)\geq {\epsilon_0^2},\ 
   \Im(y'_0)\Im(x'_1)\geq {\epsilon_0^2}.\end{equation}
Observe that \eqref{eq: more restrive contitions on de x_i'} implies condition {\it (1)} of the lemma.

By Lemma \ref{lemma: moving to good
domain} and transforming the limits $x_0,y_0,x_1,y_1$ of the integral under an appropriate element of $\Gamma$, we can
assume that 
\begin{equation}\label{eq: condition x0 x1}
\Im(x_0)\Im(x_1)\geq {\epsilon_F^2}.
\end{equation}
If $\Im(x_0)\Im(y_1)\geq {\epsilon_1^2}$, $\Im(x_1)\Im(y_0)\geq {\epsilon_1^2}$ and $\Im(y_0)\Im(y_1)\geq
{\epsilon_1^2}$ then $\epsilon(x_0,y_0,x_1,y_1)>\epsilon_0$. So we can suppose that
one of
these conditions is not satisfied. We divide the
proof in three cases, that we label  (a), (b) and  (c), according to the condition which is not satisfied.

\begin{enumerate}[(a)]
 \item If $\Im(x_0)\Im(y_1)< {\epsilon_1^2}$. Then we break the integral as
\begin{equation}\label{eq: case 3}
 \int_{x_0}^{y_0}\int_{x_1}^{y_1}\omega_f=\int_{x_0}^{y_0}\int_{x_1}^{t_1}\omega_f+\int_{x_0}^{y_0}\int_{t_1}^{y_1}
\omega_f,
\end{equation}
where $t_1\in\rho(x_1,y_1)$ satisfies that 
\begin{equation}
\label{eq: condition t1 case a}\Im (t_1)\Im(x_0)={\epsilon_1^2}.
\end{equation}
Observe that this is possible because $\Im(x_1)\geq \frac{\epsilon_F^2}{\Im(x_0)}>\frac{\epsilon_1^2}{\Im(x_0)}$ and
$\Im(y_1)<\frac{\epsilon_1^2}{\Im(x_0)}$, so that the geodesic between $x_1$ and $y_1$ in $\mathcal H_1$ intersects the
line
$\Im(z)=\frac{\epsilon_1^2}{\Im(x_0)}$, and we take $t_1$ to be the intersection point. 

If $\Im(y_0)\Im(t_1)\geq {\epsilon_0^2}$ then multiplying this condition by \eqref{eq: condition x0 x1} and
using \eqref{eq: condition t1 case a} we obtain that $\Im(x_1)\Im(y_0)\geq {\epsilon_0^2}$ and \eqref{eq: more restrive
contitions on de x_i'} holds. Moreover, dividing
\eqref{eq: condition x0 x1} by \eqref{eq: condition t1 case a} we find that
\[
 \Im(x_1)/\Im(t_1)\geq (\epsilon_F/\epsilon_1)^2
\]
so that, in view of \eqref{eq: inequality distance}, the first integral satisfies condition {\it (2)}.

If $\Im(y_0)\Im(t_1)<{\epsilon_0^2}$ then we break \eqref{eq: case 3} as
\begin{equation}\label{eq: case 1}
 \int_{x_0}^{y_0}\int_{x_1}^{y_1}\omega_f=\int_{x_0}^{t_0}\int_{x_1}^{t_1}\omega_f+\int_{t_0}^{y_0}\int_{x_1}^{t_1}
\omega_f+\int_{x_0}^{y_0}\int_{t_1}^{y_1}
\omega_f,
\end{equation}
where we take $t_0\in\rho(x_0,y_0)$ such that 
\begin{equation}\label{eq: condition t0 t1}
 \Im(t_0)\Im(t_1)={\epsilon_0^2}.
\end{equation}
This is possible because $\Im(x_0)=\frac{\epsilon_1^2}{\Im(t_1)}>\frac{\epsilon_0^2}{\Im(t_1)}$ and
$\Im(y_0)<\frac{\epsilon_0^2}{\Im(t_1)}$. Now
multiplying \eqref{eq: condition x0 x1} and \eqref{eq: condition t0 t1} and using \eqref{eq: condition t1 case a} we
obtain that $\Im(x_1)\Im(t_0)\geq{\epsilon_0^2}$. In addition, dividing \eqref{eq: condition x0 x1} by \eqref{eq:
condition t0 t1} we find that
\[
 \Im(x_0)/\Im(t_0)\geq  (\epsilon_F/\epsilon_0)^2,
\]
which implies condition {\it (2)} for the first integral.

\item If $\Im(x_1)\Im(y_0)< {\epsilon_1^2}$. This is analogous to the first case:  we break the integral as
 \begin{equation}\label{eq: case 4}
  \int_{x_0}^{y_0}\int_{x_1}^{y_1}\omega_f=\int_{x_0}^{t_0}\int_{x_1}^{y_1}\omega_f+\int_{t_0}^{y_0}\int_{x_1}^{y_1}
 \omega_f,
 \end{equation}
 where $t_0\in \rho(x_0,y_0)$ satisfies that $\Im (t_0)\Im(x_1)={\epsilon_1^2}$. If $\Im(t_0)\Im(y_1)<{\epsilon_0^2}$ then we break further the integral as
 \begin{equation}\label{eq: case 2}
  \int_{x_0}^{y_0}\int_{x_1}^{y_1}\omega_f=\int_{x_0}^{t_0}\int_{x_1}^{t_1}\omega_f+\int_{x_0}^{t_0}\int_{t_1}^{y_1}
 \omega_f+\int_{t_0}^{y_0}\int_{x_1}^{y_1}
 \omega_f,
 \end{equation}
 where we take $t_1\in \rho(x_1,y_1)$ such that  $ \Im(t_0)\Im(t_1)={\epsilon_0^2}$. 

\item If $\Im(y_0)\Im(y_1)<{\epsilon_1^2}$. In this case we can assume that 
\begin{equation}\label{eq: case c inequality}
 \Im(x_0)\Im(y_1)\geq {\epsilon_1^2}\ \text{ and }\   \Im(x_1)\Im(y_0)\geq {\epsilon_1^2},
\end{equation}

because otherwise we are in case (a) or (b). Then we can also suppose that
\begin{equation}\label{eq: condition y0 y1 case c}
 \Im(y_0)\Im(y_1)< {\epsilon_0^2},
\end{equation}
because otherwise $\epsilon(x_0,y_0,x_1,y_1)\geq \epsilon_0$. Then we can break the integral as
\[
 \int_{x_0}^{y_0}\int_{x_1}^{y_1}\omega_f=\int_{x_0}^{t_0}\int_{x_1}^{y_1}\omega_f+\int_{t_0}^{y_0}\int_{x_1}^{y_1}
\omega_f,
\]
where $t_0\in\rho(x_0,y_0)$ satisfies that 
\begin{equation}
\label{eq: condition t0 case c}\Im (t_0)\Im(y_1)={\epsilon_0^2}.
\end{equation}
Observe that
\[
 \Im(t_0)\Im(x_1)=\frac{\epsilon_0^2\Im(x_1)}{\Im(y_1)}>\Im(y_0)\Im(x_1)\geq \epsilon_1^2,
\]
so that the first integral satisfies the conditions of the Lemma, because from \eqref{eq: case c inequality} and
\eqref{eq: condition t0 case c} if follows that $\Im(x_0)/\Im(t_0)\geq
(\epsilon_1/\epsilon_0)^2=(\epsilon_F/\epsilon_1)^2.$
\end{enumerate}
\end{proof}

\subsection*{Proof of Theorem \ref{th: main}} Define the quantity $\epsilon(I)$ as
\[
\epsilon(I)=\epsilon(x_0,y_0,x_1,y_1),\text{ if } I=\int_{x_0}^{y_0}\int_{x_1}^{y_1}\omega_f.
\]
We can compute an expression $I=I_1+\dots + I_n$ with
$\epsilon(I_i)\geq\epsilon_0$ as in  \eqref{eq: eq del main theorem} by repeated application of Lemma \ref{lemma:
breaking the integrals}. To be more precise, we can apply the following algorithmic procedure. 

 We initialize a void list $V$ and a list $W=[I]$. We  denote by $W_0$ the first element in $W$. If $\epsilon(W_0)\geq
\epsilon_0$ we append $W_0$ to $V$ and we remove $W_0$ from $W$. If $\epsilon(W_0)<\epsilon_0$, we apply Lemma
\ref{lemma: breaking the integrals} to write $W_0=I_0+J_1+\dots+J_k$ with $\epsilon(I_0)\geq \epsilon_0$ and $k\leq 2$.
We append $I_0$ to $V$, we remove $W_0$ from $W$ and we append $J_1,\dots,J_k$ to $W$. Then we repeat the process to the
first element $W_0$ in the list $W$, until $W$ is void.

Observe that if this process finishes in a finite number of steps, then  at the end the list $V$ contains integrals
$I_0,I_1,\dots,I_n$ with $I=I_0+\cdots+I_n$ and $\epsilon(I_i)\geq\epsilon_0$ for all $i=1,\dots,n$, as desired. What
remains to be shown, then, is that this procedure cannot be repeated infinitely many times. As we shall now see, this
follows from properties {\it (2)} and {\it (3)} of Lemma \ref{lemma: breaking the integrals}.

First of all, observe that $I$ is the integral of $\omega_f$ along the region $\slims{x_0}{y_0}{x_1}{y_1}$. Applying
Lemma
\ref{lemma: breaking the integrals} to a certain integral $W_0=\int_{r_0}^{s_0}\int_{r_1}^{s_1}\omega_f$ amounts to give
a decomposition of the region $\slims{r_0}{s_0}{r_1}{s_1}\subseteq \slims{x_0}{y_0}{x_1}{y_1}$ into a  union of
regions with zero measure intersection. This
decomposition can be of the following  $4$ types, which correspond to \eqref{eq: case 1}, \eqref{eq: case 2}, \eqref{eq:
case 3} and \eqref{eq: case 4} respectively:
\begin{enumerate}[(i)]
\item $ \slims{r_0}{s_0}{r_1}{s_1}=\slims{r_0}{t_0}{r_1}{t_1}\sqcup \slims{t_0}{y_0}{r_1}{t_1}\sqcup
\slims{r_0}{s_0}{t_1}{s_1}$ with
$d(r_i,t_i)> d_\text{min}$ for $i=0,1$,
\item $ \slims{r_0}{s_0}{r_1}{s_1}=\slims{r_0}{t_0}{r_1}{t_1}\sqcup \slims{r_0}{t_0}{t_1}{r_1}\sqcup
\slims{t_0}{s_0}{r_1}{s_1}$ with
$d(r_i,t_i)> d_\text{min}$ for $i=0,1$,
\item $ \slims{r_0}{s_0}{r_1}{s_1}=\slims{r_0}{s_0}{r_1}{t_1}\sqcup \slims{r_0}{s_0}{t_1}{s_1}$ with
 $d(r_1,t_1)> d_\text{min}$, or
\item $ \slims{r_0}{s_0}{r_1}{s_1}=\slims{r_0}{t_0}{r_1}{s_1}\sqcup\slims{t_0}{s_0}{r_1}{s_1}$ with
 $d(r_0,t_0)> d_\text{min}$.
\end{enumerate}

Each time that an application of Lemma \ref{lemma: breaking the integrals} gives a decomposition of type (i) or (ii), 
the first term is a subregion of $\slims{x_0}{y_0}{x_1}{y_1}$ of area at least $d_\text{min}^2$. Since the area of
$\slims{x_0}{y_0}{x_1}{y_1}$ is finite,
this cannot happen infinitely many times. Therefore, in the algorithmic process described above we can assume that,
after a finite number of steps, all applications of  Lemma \ref{lemma: breaking the integrals} give rise to
decompositions of types (iii) or (iv). 

At this stage, the algorithmic procedure applied to an integral
$W_0$ is as follows. Let $R_0=\slims{r_0}{s_0}{r_1}{s_1}$ be the region associated to $W_0$, and let
$\Delta(R_0)=d(r_0,s_0)+d(r_1,s_1)$. As application of Lemma~\ref{lemma:
breaking the integrals} one obtains a decomposition $R_0=R_1\sqcup S_1$ of type (iii) or (iv). Observe that
$\Delta(S_1)\leq \Delta(R_0)-d_\text{min}$ by property {\it (3)} of Lemma~\ref{lemma: breaking the integrals}. If
$\epsilon(S_1)\geq
\epsilon_0$ then procedure for $W_0$ ends; otherwise one iterates by applying again \ref{lemma: breaking the
integrals}, obtaining $S_1=R_2\sqcup S_2$ with $\Delta(S_2)\leq \Delta(S_1)- d_\text{min}$. It is clear that this procedure
cannot continue indefinitely because $\Delta(S_i)\geq 0$. Therefore, at some step $\epsilon(S_i)\geq \epsilon_0$ and
the process for $W_0$ ends.

\section{Dependence on the continued fractions}\label{sec: Dependence on the continued fractions}
The computation of ATR points is equivalent to the computation of semi-definite integrals of the form
\begin{equation}\label{eq: semidefinite integral}
 \int^{\tau_0}\int_\infty^{c}\omega_f^+, \ \ c\in F.
\end{equation}
As we recalled in \S \ref{subsection: continued fractions trick}, one can write \eqref{eq: semidefinite integral} as a
sum of ordinary $4$-limit integrals by expressing $c$ as a continued fraction with coefficients in $\cO_F$. If the
limits of the resulting ordinary integrals are too close to the real axis then the number of terms to sum for a
prescribed accuracy, and therefore the number of Fourier coefficients to be computed, may be too large. In this case,
the algorithm described in \S \ref{sec: Integrals of Hilbert Modular forms} can be used to express them as sums of
integrals whose limits are uniformly bounded away from the real axis, reducing the number of terms and Fourier
coefficients needed.

If  $F$ is norm-euclidean then the  euclidean algorithm gives an effective procedure for computing continued
fractions. This is the method used in the numerical calculations over $\Q(\sqrt{29})$,
$\Q(\sqrt{37})$ and $\Q(\sqrt{41})$ carried out in \cite{DL}. But this can only be done in a few 
fields: there are only finitely many norm-euclidean real quadratic fields, being $\Q(\sqrt{73})$ the one having largest
discriminant.

 An  algorithm for computing continued fractions in $2$-stage euclidean real quadratic fields was given in \cite{GM}. A
field $F$ is said to be \emph{$2$-stage euclidean} if for every $a,b\in \cO_F$, $b\neq 0$, there exist either:
\begin{enumerate}[(i)]
 \item $q,r\in \cO_F$ with $a=qb+r$ and $\Nm(r)<\Nm(b)$, or
\item \label{eq: 2-stage chain} $q_1,q_2,r_1,r_2\in \cO_F$ with 
$
 a=q_1 b+ r_1,\ \ b=q_2r_1+r_2
$
and $\Nm(r_2)<\Nm(b)$. 
\end{enumerate}

All real quadratic fields of class number $1$ are conjectured to be $2$-stage euclidean (see \cite{cooke}).
Actually, the algorithm of \cite{GM} can also be used to verify that a given $F$ is $2$-stage
euclidean, and this was used to prove that all real quadratic fields of class number $1$ and discriminant up to $8,000$
are indeed $2$-stage euclidean \cite[Theorem 4.1]{GM}.

Unlike the situation encountered in norm-euclidean fields, $2$-stage division chains as in condition \eqref{eq: 2-stage
chain} are not unique. As a consequence,  elements of $F$ admit in general many different continued fraction expansions. This
leads to different expressions of \eqref{eq: semidefinite integral} as a sum of ordinary integrals, whose limits may
have very different imaginary parts. As it will be illustrated
in \S \ref{sec: Numerical verification of Darmon's conjecture} with some explicit examples, numerical experiments
suggest that it is useful to exploit non-uniqueness of continued fractions, by searching for continued fractions
 leading to integrals whose limits have large imaginary parts; or, to be more precise, whose limits give
large values of the quantity $\epsilon$ defined in  \eqref{eq: def de epsilon}. The procedure for computing an
expression such as \eqref{eq: semidefinite integral} is then:
\begin{enumerate}[(1)]
 \item Compute all continued fraction expansions of $c$ given by the algorithm of \cite{GM}, which have length up to a
certain fixed bound. 
\item For each continued fraction, compute the corresponding expression of \eqref{eq: semidefinite integral} as a sum
of ordinary integrals and compute $\epsilon_{\text{min}}$: the minimum of the quantities $\epsilon$ corresponding to the
limits. 
\item\label{step: high eps} Choose the continued fraction giving the highest $\epsilon_{\text{min}}$.
\item\label{step: last} For each of the ordinary integrals appearing in the expression given by the continued fraction
found in the
previous step, compute the quantity $\epsilon$ and apply the algorithm of Theorem~\ref{th: main} if
$\epsilon<\epsilon_F$, with a suitable choice of $\epsilon<\epsilon_0<\epsilon_F$.
\end{enumerate}

We end the section with some remarks about the algorithm above.

\begin{enumerate}[1.]
\item One can exploit non-uniqueness of $2$-stage division chains even if $F$ is euclidean. As the next section
illustrates, this may be beneficial since it usually gives rise to integrals with larger values of $\epsilon$,
providing an improvement even on the curves already considered in \cite{DL}. In some cases, it may even happen that the
$\epsilon_{\text{min}}$ obtained in this way is higher than $\epsilon_F$, in which case it is not necessary to apply the
algorithm provided by Theorem~\ref{th: main}. However, the lack of an a priori estimate of the value
$\epsilon_{\text{min}}$ obtained by the non-uniqueness of division chains trick explains the key importance of
Theorem~\ref{th: main} in treating the cases where $\epsilon_{\text{min}}<\epsilon_F$.

\item In Step (\ref{step: high eps}) we choose the continued fraction giving the highest $\epsilon_{\text{min}}$ 
because experimentally this seems to produce the fewer resulting integrals in step $(4)$. We have no rigorous
explanation
for this fact, although it seems reasonable that better initial conditions give better results.
\item  There is a trade-off between small and
large values of $\epsilon_0$ in Step (\ref{step: last}) above: smaller values yield less
integrals after the breaking process, but each of these integrals requires more Fourier coefficients at the time of
integration; on the other
hand, higher values lead to integrals requiring less Fourier coefficients, but the number of resulting integrals tends
to be higher. Experimentally, we found that the running time of the algorithm is more sensible to the number of needed
Fourier coefficients, so a value of $\epsilon_0$  close to
$\epsilon_F$ seems to be a good choice. For instance, in the implementation of the algorithm used to compute the
numerical examples of the next section, we used $\epsilon_0= 0.81\epsilon_F$, which corresponds to a value of
$\epsilon_1= 0.9\epsilon_F$.
  \end{enumerate}

\section{Numerical verification of Darmon's conjecture}\label{sec: Numerical verification of Darmon's conjecture}

In this section we illustrate the  algorithm described above by calculating approximations to ATR points which add
numerical evidence on top of the one presented in~\cite{DL}. Before detailing the computation of an ATR
point on $E_{509}$, we comment on some calculations of ATR points on three $\Q$-curves that we denote $E_{29}$,
$E_{37}$, 
and $E_{109}$. The curves $E_{29}$ and $E_{37}$ were also considered in \cite{DL}, and we have included them here in
order to compare the computational requirements of the algorithm used in \cite{DL} with the one
presented in this note. The curve $E_{109}$ is an example of a  curve of conductor $1$ defined over a real quadratic
field of
class number $1$ which is not norm-euclidean. Therefore, it was not numerically accessible before, although it is a
$\Q$-curve and algebraic points  can be more efficiently computed by using the
Heegner point method of \cite{DRZ}.

The computations for $E_{29}$, $E_{37}$, and $E_{109}$ were performed on a laptop with \emph{Intel Core\texttrademark{}
i5-2540M} CPU running at $2.60\text{ GHz}$, and $8\text{ GB}$ of memory. For the curve $E_{509}$ we used a
machine equipped with eight \emph{Quad-Core AMD Opteron\texttrademark{} Processor 8384} for a total of $32$ cores running
each at $800\text{ MHz}$, and equipped with $320\text{ GB}$ of memory.
\subsection*{The curve $E_{29}$}
Consider the curve defined over $F=\Q(\sqrt{29})$ and given by the equation
\[
E_{29}\colon y^2 + xy + (5\omega +11) y = x^3,\quad \omega = \frac{1+\sqrt{29}}{2}.
\]
For this field the estimated  $C_F$ is $C_F \simeq 5$ (see Remark~\ref{remark: explicit constants}), which by
\eqref{eq: true epsilon_F} yields $\epsilon_F\simeq 0.0736$.

We consider the ATR field $K=F(\beta)$ with $\beta = \sqrt{9\omega +3}$, for which $E_{29}(K)$ has a non-torsion point
with $x$-coordinate equal to $-1/3$. With the algorithm used in~\cite{DL}, one obtains integrals with
$\epsilon_{\text{min}} \simeq 0.00145$. In order to get
$12$ decimal digits, which is the minimum precision in which the calculations in~\cite{DL} were performed, one would
have needed to find the Fourier  coefficients of all ideals up to norm $N\simeq 6.7\cdot 10^7$. 

Using the non-uniqueness of continued fraction expansions as explained in Section~\ref{sec: Dependence on the continued
fractions}, considering expansions of length up to $5$, we obtained $5$ integrals with imaginary part
$\epsilon_{\text{min}}\simeq
0.0072$, which is almost $5$ times better than before. In order to obtain the same precision one would have to find
the Fourier coefficients of ideals up to norm $N\simeq 2.7\cdot 10^6$, which is almost $25$ times less.

 Since
$\epsilon_{\text{min}}<\epsilon_F$,  we broke further the integrals with the algorithm of Theorem~\ref{th: main} to move
the imaginary parts of the limits close to this theoretical optimal, with a choice of $\epsilon_0=0.81\epsilon_F$
in Step (4) of the algorithm outlined in Section~\ref{sec: Dependence on the continued fractions}. This yielded $539$
integrals with an imaginary part of
$\epsilon_{\text{min}}\simeq 0.0596$, and allowed us to obtain the same precision of $12$ digits by only considering
ideals up to norm $N\simeq 40,000$: an improvement by a factor of $1,675$. By taking ideals of norm up to $40,000$ we
obtained that
\[
J_\tau = 13.2923360157968468468\ldots - 10.78402031269077180934\ldots i,
\]
and $-3\cdot J_\tau$ coincides with  $P$ up to the prescribed accuracy. The calculation took less than two minutes.

\subsection*{The curve $E_{37}$}
Let $E_{37}$ be the curve defined over $F = \Q(\sqrt{37})$ appearing in~\cite{DL} and having equation
\[
E_{37}\colon y^2+y=x^3+2x^2-(19+8\omega)x + 28+11\omega,\quad \omega = \frac{1+\sqrt{37}}{2}.
\]
For this field the constant $C_F$ is approximately equal to $6$. By \eqref{eq: true epsilon_F} we see that
$\epsilon_F\simeq 0.044$. 

We consider the field $K=F(\beta)$, with $\beta = \sqrt{4\omega+10}$, and one of the points computed
by~\cite{DL}, namely
\[
P = (-\beta^2/8-3/4, -\beta^3/8-1/2).
\]
Using the algorithm of \cite{DL} one obtains a minimal imaginary part of $\epsilon_{\text{min}}\simeq 0.0012$, which
means that one has to integrate using the Fourier coefficients up to norm $N\simeq 1.12\cdot 10^8$ for obtaining $12$
digits of precision. 

In order to illustrate the algorithm of Theorem \ref{th: main} we rewrote the integrals provided by the method
of \cite{DL} as a sum involving $328$ integrals (with a choice of $\epsilon_0=0.81\epsilon_F$). The minimal imaginary
part then improved to
$\epsilon_{\text{min}}\simeq \epsilon_0\simeq 0.0359$, which means that to get $12$ digits of precision it is
enough to use ideals of
norm up to $N\simeq 138,000$. This is an improvement by a factor of $815$. In this
case, it took about $7$ minutes to find that
\[
J_\tau = -1.3589031642485772101\ldots + 8.36575277665729384437\ldots i,
\]
which satisfies the equality
\[
5 J_\tau \stackrel{?}{=} -8 P
\]
up to the prescribed accuracy.

\subsection*{The curve $E_{109}$}
In the two remaining subsections we present larger examples that were not available to~\cite{DL}.  First, consider the
curve $E_{109}$ defined over the field $F = \Q(\sqrt{109})$, and given by the equation
\[
E_{109}\colon y^2 + \omega xy = x^3 -(1+\omega)x^2 - (58\omega+245) x -630\omega-2944, \quad
\omega= \frac{1+\sqrt{109}}{2}
\]
Although $E_{109}$ is a $\Q$-curve, the field $F$ is not norm-euclidean and therefore this example was not available
before. For this field we have that $C_F\simeq 10.4$, giving that $\epsilon_F\simeq
0.006$. By using the algorithm of Theorem~\ref{th: main} with, say, $\epsilon_0=0.81\epsilon_F$, one can express any
integral $\int_{x_0}^{y_0}\int_{x_1}^{y_1}\omega_f^+$ as a sum of integrals with $\epsilon\simeq
0.81\cdot\epsilon_F\simeq 0.0048$. In order to compute any such integral with $\epsilon=0.0048$ to a precision of $12$
digits, one needs to sum the
Fourier coefficients of norm up to roughly $2\cdot 10^7$.

Let us  consider the point $P = (3\omega +11, \frac 12\beta - 7\omega-81/2 )$ defined over the field
$K=F(\beta)$, with $\beta=\sqrt{268\omega+1265}$. Exploiting the non-uniqueness of continued fractions, and considering
expansions of length up to $5$, we obtained $8$ integrals with $\epsilon_{\text{min}}\simeq 0.035$, which is roughly
$6$ times higher than $\epsilon_F$. It is not necessary then in this case to further break the integrals. We computed
an approximation to the ATR point by considering ideals of norm up to $N = 430,000$, obtaining that
\[
J_\tau = -3.24024368505944150\ldots\cdot 10^{-12} - 42.392087963225793791\ldots i,
\]
which satisfies
\[
J_\tau \stackrel{?}{=} -2 P
\]
up to the prescribed precision of $12$ digits. This computation took less than $3$ minutes.

\subsection*{The curve $E_{509}$}
We consider here the curve already mentioned in Section~1 and in~\cite{DL}, defined over $F = \Q(\sqrt{509})$ and
given by the equation
\[
 y^2-xy-\omega y=x^3+(2+2w)x^2+(162+3w)x+(71+34\omega),\quad \omega=\frac{1+\sqrt{509}}{2}.
\]
We have that $C_F\simeq 22.5$, and therefore $\epsilon_F\simeq 0.0015$. Theorem~\ref{th: main} allows us to express any
integral $\int_{x_0}^{y_0}\int_{x_0}^{y_0}\omega_f^+$ as a sum of integrals having $\epsilon<\epsilon_0<\epsilon_F$. For
instance,
for a choice of $\epsilon_0=0.81\epsilon_F$ we obtain that each of those integrals could be computed to $12$ digits of
precision by summing over the Fourier coefficients of norm up to roughly $1.6\cdot 10^9$; a bound which, although being
large, is within
reach of the current technology.

We consider the ATR field $K = F(\beta)$ where $\beta = \sqrt{9144\omega + 98577}$, and the point $P\in E(K)$
of
infinite order having coordinates
\[
P= (\omega + 17 , \beta/2 + \sqrt{509}/2 + 9 ).
\]
The field extension $K/F$ has relative discriminant of norm $55$, which is relatively small. Write
\[
\cO_K = \cO_F + \alpha\cO_F,\quad \alpha^2 + \alpha = 127\sqrt{509} + 2865.
\]
The ATR points are conjectured to be defined over the Hilbert class field of $K$. Since $K$ has class number $2$, we
will need to  compute the points corresponding to the two non-equivalent optimal embeddings if we want to obtain a point
defined over $K$. The first of these embeddings maps
\[
\alpha \mapsto \varphi_1(\alpha)=\mtx{0}{254\omega + 2738}{1}{-1},
\]
whereas the second maps
\[
\alpha\mapsto \varphi_2(\alpha)=\mtx{0}{127\omega +1369}{2}{-1}.
\]
The fixed points for the induced action of $K^\times$ on $\cH$ given by the embedding $v_0\colon K\hookrightarrow  \C$
are, respectively:
\begin{align*}
\tau_0^{(1)}&= 0.5 + 0.024492046328012136937583\dots i\\
\tau_0^{(2)}&=\frac 12 \tau_0^{(1)}.
\end{align*}
Exploiting the non-uniqueness of quadratic continued fractions we obtain $4$ integrals for the first of the points, and
$8$ for the second. The then minimal imaginary parts are $\epsilon_{\text{min}}^{(1)}=0.01917$ and
$\epsilon_{\text{min}}^{(2)}= 0.002926$. Since $\epsilon_{\text{min}}^{(1)}>\epsilon_F$ and
$\epsilon_{\text{min}}^{(2)}>\epsilon_F$ we see that in this case it is not necessary to break further the integrals
using Theorem~\ref{th: main}.

In order to obtain about $12$ decimal digits of accuracy  we precomputed
the Fourier coefficients of all ideals up to norm $4\cdot 10^8$. The total computation time was under two days on the
$32$-processor machine specified at the beginning of this section. We should note that in this computation we heavily
exploited parallelism, both when computing the Fourier coefficients as well as during the integration step. The period
lattices for $E_K$ attached to the N\'eron differential $\omega_{E_K} = \frac{dx}{2y-x-\omega}$ are
\begin{align*}
\Lambda_0 &= \langle -5.38425378853615683456\dots, -7.44383552310672504690\dots i \rangle = \langle
\lambda_0^+,\lambda_0^-\rangle,\\
\Lambda_1 &= \langle 2.47855898378449003059\dots, 1.14589256545011559322\dots i\rangle = \langle
\lambda_1^+,\lambda_1^-\rangle.
\end{align*}

A preimage of $P$ on $\C/\Lambda_0$ under the Weierstrass map is
\[
z = -2.6921268942680784172834\dots - 5.1426086531573572370822\dots i.
\]

The computed values are
\begin{align*}
J_{\tau}^{(1)} &=  22.63291528772669504213102498\ldots i,\\
J_\tau^{(2)} &  = 106.761524788388057098773188\ldots - 73.6179507973347981534240251\ldots i.
\end{align*}
Setting $J_\tau = J_\tau^{(1)}+J_\tau^{(2)}$ we find that:
\[
\Big|\frac{J_\tau}{\lambda_1^+} - 4 z +10 \lambda_0^+\Big|\simeq 5.126285 \cdot 10^{-11},
\]
which suggests that
\[
J_\tau \stackrel{?}{=} 4 z.
\]

\bibliographystyle{amsalpha}
\bibliography{refs}

\end{document}